\documentclass{conm-p-l}

\usepackage{amsthm,amsfonts, cite}
\usepackage{amssymb,graphicx,color}
\usepackage[all]{xy}
\usepackage{mathpazo}
\usepackage{eucal}
\usepackage{enumerate}
\usepackage{amsmath}
\usepackage{hyperref}
\usepackage{tikz-cd}
\numberwithin{equation}{section}

\theoremstyle{plain}
\newtheorem{teo}{Theorem}[section]

\newtheorem*{teo*}{Theorem}
\newtheorem{teoA}{Theorem}

\newtheorem*{cor*}{Corollary}
\newtheorem{corA}[teoA]{Corollary}

\newtheorem{lem}[teo]{Lemma}
\newtheorem*{lem*}{Lemma}
\newtheorem{prop}[teo]{Proposition}
\newtheorem*{prop*}{Proposition}

\theoremstyle{remark}

\newcommand{\R}{\ensuremath{{\mathbb{R}}}}
\newcommand{\esf}{\ensuremath{{\mathbb{S}}}}
\newcommand{\g}{\ensuremath{\mathtt{g}}}

\newcommand\tOmega{\ensuremath{\widetilde{\Omega}}}
\newcommand\tD{\ensuremath{\widetilde{\mathcal{D}}}}

\begin{document}

\title[Magnetic flatness and E. Hopf's theorem for magnetic systems]{Magnetic flatness and E. Hopf's theorem for magnetic systems}

\author[V. Assenza]{Valerio Assenza}
\author[J. Marshall Reber]{James Marshall Reber}
\author[I. Terek]{Ivo Terek}

\address{Instituto de Matem\'{a}tica Pura e Aplicada, Rio de Janeiro, RJ 22460-320, Brazil}
\email{valerio.assenza@impa.br}

\address{Department of Mathematics, The Ohio State University, Columbus, OH 43210, USA}
\email{marshallreber.1@osu.edu}
\email{terekcouto.1@osu.edu}


\begin{abstract}
	Using the notion of magnetic curvature recently introduced by the first author, 
	we extend E. Hopf's theorem to the setting of magnetic systems. Namely, we prove that if the magnetic flow on the $s$-sphere bundle is without conjugate points, then the total magnetic curvature is non-positive, and vanishes if and only if the magnetic system is magnetically flat. We then prove that magnetic flatness is a rigid condition, in the sense that it only occurs when either the magnetic form is trivial and the metric is flat, or when the magnetic system is K\"ahler, the metric has constant negative sectional holomorphic curvature, and $s$ equals the Ma\~{n}\'{e} critical value.
\end{abstract}


\maketitle

\section{Introduction} \label{sec:introduction}

A classical rigidity theorem in Riemannian geometry is due to E.\ Hopf \cite{Hopf_1948}: \emph{if a closed Riemannian surface is without conjugate points, then its total Gaussian curvature is non-positive, and vanishes if and only if the metric is flat.} The Gauss-Bonnet theorem then implies that every metric without conjugate points on the two-di\-men\-si\-o\-nal torus $\mathbb{T}^2$ is flat. Hopf's theorem was generalized to higher dimensions by Green \cite{Green_1958}, with the scalar curvature replacing the Gaussian curvature. However, in absence of a higher dimensional argument \`{a} la Gauss-Bonnet, the question of whether a metric without conjugate points on a torus $\mathbb{T}^n$ is flat remained open for decades. The problem was finally solved using a different array of techniques by Burago and Ivanov \cite{Burago-Ivanov_1994}.

In this paper, we generalize Green's argument to the setting of magnetic systems without conjugate points. This generalization naturally leads us to studying flatness in a magnetic sense. 
A \emph{magnetic system} on a closed smooth manifold $M$ is a pair $(\g,\sigma)$, where $\g = \langle \cdot, \cdot \rangle$ is a Riemannian metric and $\sigma$ is a closed $2$-form on $M$. In this context, $\sigma$ is referred to as the \emph{magnetic form}, and the requirement that $\sigma$ is closed corresponds to the well-known fact that magnetic fields are divergence-free, courtesy of Maxwell's equation. The skew-adjoint endomorphism $\Omega$ of $TM$ corresponding to $\sigma$ under $\g$ is referred to as the \emph{Lorentz force}, and it is defined by
\begin{equation}\label{eqn:Lorentz_force_def}
\sigma(v,w) = \langle v,\Omega(w) \rangle \text{ for all }v, w \in TM.
\end{equation}

Recall that a magnetic system is the mathematical formalism used to describe the motion of a charged particle moving on a Riemannian manifold under the action of a magnetic force field. More precisely, a trajectory $\gamma\colon \R \to M$ of a charged particle is described by the second-order ordinary differential equation
\begin{equation}\label{eqn:Landau-Hall}
\frac{{\rm D}\dot{\gamma}}{{\rm d}t} = \Omega(\dot{\gamma}),
\end{equation}
where ${\rm D}/{\rm d}t$ denotes the covariant derivative along $\gamma$ induced by the Levi-Civita connection $\nabla$ of $\g$. Solutions of \eqref{eqn:Landau-Hall} are called \emph{$(\g,\sigma)$-geodesics}. Lifting \eqref{eqn:Landau-Hall} to $TM$, we obtain the \emph{magnetic flow} $\varPhi^{(\g,\sigma)}\colon\R\times TM\to TM$ associated with $(\g,\sigma)$. Observe that in absence of a magnetic interaction (that is, when $\sigma = 0$), \eqref{eqn:Landau-Hall} reduces to the standard equation for geodesics, and $\varPhi^{(\g,0)}$ is the geodesic flow of the metric $\g$. Since $(\g,\sigma)$-geodesics have constant speed, the infinitesimal generator of $\varPhi^{(\g,\sigma)}$ is tangent to the \emph{$s$-sphere bundles} $\Sigma_s = \{ v\in TM \mid \|v\|=s \}$ for $s > 0$. Thus, $\Sigma_s$ is an invariant set for the flow $\varPhi^{(\g,\sigma)}$, and hereafter we denote by ${\varPhi^{(\g,\sigma)}_s : \mathbb{R} \times \Sigma_s \rightarrow \Sigma_s}$ the restriction of $\varPhi^{(\g,\sigma)}$ to $\Sigma_s$. The dynamics of magnetic systems have been studied for many decades; for some historical background, see \cite{Arnold_1961, AnosovSinai_1967, Ginzburg_2004}.

Linearizing \eqref{eqn:Landau-Hall} along a $(\g,\sigma)$-geodesic $\gamma$, we obtain the differential equation
\begin{equation}\label{eqn:magnetic_Jacobi_equation}
  \frac{{\rm D}^2J}{{\rm d}t^2} + R(J,\dot{\gamma})\dot{\gamma} = (\nabla_J\Omega)(\dot{\gamma}) + \Omega\left(\frac{{\rm D}J}{{\rm d}t}\right)
\end{equation}for vector fields $J$ along $\gamma$. A solution $J$ of \eqref{eqn:magnetic_Jacobi_equation} is called a \emph{\hbox{$(\g,\sigma)$-Jacobi} field}. In addition, a $(\g, \sigma)$-Jacobi field $J$ is \emph{normal} if it satisfies  
\begin{equation} \label{eqn:normal} 
\left\langle  \frac{{\rm D}J}{{\rm d}t } , \dot{\gamma} \right\rangle = 0.
\end{equation}
A point $\gamma(\tau)$ is said to be \emph{conjugate} to $\gamma(0)$ along $\gamma$ if there exists a nontrivial normal Jacobi field $J$ along $\gamma$ such that $J(0) = 0$ and $J(\tau)$ is proportional to $\dot{\gamma}(\tau)$. A $(\g,\sigma)$-geodesic $\gamma$ is \emph{without conjugate points} if $\gamma(\tau)$ is not conjugate to $\gamma(0)$ for every $\tau \neq 0$, and the flow $\varPhi^{(\g, \sigma)}_s$ is \emph{without conjugate points} if all of $(\g, \sigma)$-geodesics with speed $s$ are without conjugate points.\footnote{The above definition is related to critical values of a suitable exponential map defined for magnetic systems. For more details, see \cite{Herreros_2012, AdachiBai_2013}.}

In \cite{Assenza_2023}, the first author introduced a notion of magnetic curvature related to the second variation of the action which extends to higher dimensions the well-known Gaussian magnetic curvature introduced by G.\ and M.\ Paternain in \cite{PP96} for surfaces. We briefly recall its definition. Let $E$ and $E^1$ be the bundles of complementary and unit-com\-ple\-men\-ta\-ry directions over the unit sphere bundle $SM$, whose fibers over $v\in SM$ are given by $E_v=v^\perp$ and $E^1_v = v^\perp\cap SM$. Denote by $R$ the Riemann curvature tensor, and for $s > 0$, consider the bundle endomorphisms $R^\Omega_s, A^\Omega : E \rightarrow E$ given by 
\begin{align} 
	  (R^\Omega_s)_v(w) &= s^2 R(w,v)v - s(\nabla_w\Omega)(v)  + \frac{s}{2}\left( (\nabla_{v}\Omega)(w) - \langle (\nabla_{v}\Omega)(w), v \rangle v \right), \label{eqn:ROmega} \\
	 (A^\Omega)_v(w) &= \frac{3}{4}\langle w, \Omega(v)\rangle \Omega(v) - \frac{1}{4} \Omega^2(w) - \frac{1}{4}\langle \Omega(w),\Omega(v)\rangle v. \label{eqn:AOmega}
\end{align}

The \emph{magnetic curvature operator at level $s$} is the section $M^\Omega_s$ of ${\rm End}(E)$ given by
\begin{equation}\label{def:MOmega}
  (M^\Omega_s)_v(w) = (R^\Omega_s)_v(w) + (A^\Omega)_{v}(w).
\end{equation}
The \emph{magnetic sectional curvature} ${\rm sec}^\Omega_s\colon E^1\to \R$ and the \emph{magnetic Ricci curvature} ${\rm Ric}^\Omega_s\colon SM\to \R$ \emph{at level $s$} are given by, respectively:
\begin{equation}\label{def:sec_and_Ric_mag}
 ({\rm sec}^\Omega_s)_v(w) = \left \langle (M^\Omega_s)_{v}(w),w\right\rangle  \quad\mbox{and}\quad {\rm Ric}^\Omega_s(v) = {\rm tr}\,(M^\Omega_s)_{v}.
\end{equation}

We note that the magnetic curvature function captures relevant dynamical information for the magnetic system. For instance, it can be deduced from \cite[Theorem 4.1]{Wojtkowski_2000} that if ${\rm sec}^\Omega_s < 0$, then $\varPhi^{(\g, \sigma)}_s$ is Anosov.\footnote{See also Appendix \ref{appendixAnosov} for a proof.} Under the assumption that ${\rm Ric}^\Omega_s > 0$, it was shown in \cite{Assenza_2023} 
that there exists a closed contractible $(\g, \sigma)$-geodesic for small $s$.\footnote{See \cite{BT98} for previous results in this direction.}

In Proposition \ref{propn:curvature-identity}, we relate $M^\Omega_s$ to \eqref{eqn:magnetic_Jacobi_equation}, and in Proposition \ref{prop:negative_sec_no_conj} we show that if ${\rm sec}^\Omega_s \leq 0$, then $\varPhi^{(\g, \sigma)}_s$ is without conjugate points. As in the Riemannian setting, the converse does not hold in general. An example of a magnetic flow without conjugate points whose magnetic curvature changes sign was constructed \cite[Section 7]{Burns-Paternain_2002}.
Nevertheless, our main result gives a partial converse, generalizing \cite[Theorem 1]{Gouda_1998} from the case where $\nabla \sigma = 0$:

\begin{teoA}\label{teo:magnetic-Hopf}
	Let $(\g,\sigma)$ be a magnetic system on a closed manifold $M$, and let $s> 0$. If $\varPhi^{(\g,\sigma)}_s$ is without conjugate points, then
	\begin{equation}
		\int_{SM} {\rm Ric}^\Omega_s(v) \,{\rm d}\mu_\g(v) \leq 0,
	\end{equation}
where ${\rm d}\mu_{\g}$ denotes the Liouville measure on $SM$.
Moreover, equality holds if and only if $M^\Omega_s=0$. \end{teoA}

The strategy of the proof is similar to the one given in \cite{Green_1958}. With the aid of an appropriate twisted connection introduced in Section \ref{sec:anisotropic}, we construct the \emph{magnetic Green bundles} along $(\g,\sigma)$-geodesics without conjugate points in Section \ref{sec:mag_green_bundles}. The existence of such bundles guarantees that we are able to find non-trivial solutions to a suitable \emph{magnetic Riccati equation} (see \eqref{eqn:Riccati-t}), and the result follows quickly. We note that similar geometric constructions were previously introduced in the classical Riemannian case by Green in \cite{Green_1958}, for Finsler metrics without conjugate points in \cite{Foulon_1992}, and in the general context of Hamiltonian dynamics in \cite{ContrerasIturriaga_1999, MarieClaudeArnaud_2008, BialyMacKay_2004}. 

We now observe some consequences of Theorem \ref{teo:magnetic-Hopf}. Using the linearity of $\Omega$, one can deduce
\begin{equation} \label{eqn:integral_RO} \int_{SM} {\rm tr}(R^\Omega_s)_v \,{\rm d} \mu_\g(v) = s^2 \frac{{\rm vol}(\esf^{n-1})}{n} \int_M {\rm scal}(x) \, {\rm d}\nu_\g(x), \end{equation}
where ${\rm d}\nu_g$ is the Riemannian volume measure on $M$.
Combining \eqref{eqn:integral_RO} with the fact that ${\rm tr}(A^\Omega)_v$ is always non-negative and is positive on some open subset of $SM$ as soon as $\sigma$ is not identically zero (see \cite[Lemma 5]{Assenza_2023} or \eqref{eqn:AOmega}), the following corollary holds:
\begin{corA} \label{cor:conjugate_point}   If $\sigma$ is not identically zero, then for every sufficiently small $s$ the flow $\varPhi^{(\g, \sigma)}_s$ has at least one pair conjugate points. Moreover, if $\g$ is a Riemannian metric on $M$ such that 
		\[\int_M {\rm scal}(x) \, {\rm d}\nu_\g(x) > 0,\]
		then for every magnetic form $\sigma$ and every $s > 0$ the flow $\varPhi^{(\g, \sigma)}_s$ has at least one pair conjugate points.
\end{corA}

This can be interpreted as saying that if the magnetic deformation is strong, then the magnetic flow must have at least one pair conjugate points. Naturally, this leads us to the following question: \emph{given a Riemannian metric $\g$ with at least one pair conjugate points, does there exist a closed $2$-form $\sigma$ and a value $s$ such that the corresponding flow $\varPhi^{(\g, \sigma)}_s$ is without conjugate points?} While Corollary \ref{cor:conjugate_point} points to a negative answer for small values of $s$, the question is open to our best knowledge.
Some progress towards this question has been made by Bialy in \cite{Bialy_2000}, where it is shown that if $\g$ is conformally flat on the $n$-torus $\mathbb{T}^n$, then for every $s$ the magnetic flow $\varPhi^{(\g, \sigma)}_{s}$ is without conjugate points if and only if $\sigma$ is identically zero and the metric is flat. Bialy also conjectured that such a result should hold for arbitrary metrics, which would generalize \cite{Burago-Ivanov_1994} to the magnetic case. 

If $M$ is an orientable surface, then every closed $2$-form $\sigma$ is related to the area form $\nu_\g$ by a smooth function $b$, in the sense that $\sigma = b \nu_\g$. We denote the magnetic system $(\g, b \nu_g)$ by $(\g, b)$. In this setting, the magnetic curvature functions ${\rm sec}^\Omega_s$ and ${\rm Ric}^\Omega_s$ reduce to the \emph{Gaussian magnetic curvature at level $s$}, which is given by 
\begin{equation*}\label{eqn:mag_Gaussian}
	 K^{\g,b}_s(x,v) = s^2K^{\g}(x) - s\,{\rm d}b_x({\rm i}v) + b(x)^2;
\end{equation*}
see  \cite[Lemma 8]{Assenza_2023}.
Here, ${\rm i}v$ is the vector obtained by rotating $v$ by $\pi/2$ according to the fixed orientation, and $K^{\g}\colon M\to\R$ is the Gaussian curvature. As a consequence of the Gauss-Bonnet theorem, we have
\[ \int_{SM} K^{\g, b}_s(x,v) \, {\rm d} \mu_\g(x,v) = 2\pi \left[s^2 \chi(M) + \int_M b^2(x) \, {\rm d}\nu_{\g}(x) \right],\]
where $\chi(M)$ is the Euler characteristic of $M$. As pointed out in Corollary \ref{cor:conjugate_point}, every magnetic flow on $\esf^2$ has conjugate points. In the case of a magnetic flow on $\mathbb{T}^2$, we are able to recover \cite[Theorem 1]{Bialy_2000}:
\begin{corA}
	The only magnetic flows on $\mathbb{T}^2$ without conjugate points are of the form $\varPhi^{(\g_0, 0)}_s$, where $\g_0$ is a flat metric.
\end{corA}

Summarizing, the genus of the surface must be at least two for there to be a non-trivial magnetic system without conjugate points. If we assume that $(\g, b)\neq (\g_0,0)$, then $K^{\g, b}_s \equiv 0$ if and only if $b$ is constant, $\g$ has constant negative curvature $k$, and $s = |b|/\sqrt{-k}$. Observe that if this happens, then 
 $\varPhi^{(\g, b)}_s$ is the horocycle flow \cite[Section 5.2]{CieliebakFrauenfelderPaternain_2010}. In particular, 
this allows us to view the horocycle flow on hyperbolic surfaces as the magnetic analogue of $\mathbb{T}^2$ equipped with a flat metric. 

In order to bring the discussion to higher dimensions, we recall that the \emph{Ma\~{n}\'{e} critical value} $s_0 = s_0(\g, \sigma) \in [0,+\infty]$ of the pair $(\g, \sigma)$ is defined as
\begin{equation} \label{eqn:mane}
	s_0 = \begin{cases} {\displaystyle \inf_{d\theta =  p^*\sigma} \sup_{x \in \tilde{M}}} \|\theta_x\| & \text{if } p^* \sigma \text{ is exact}, \\ +\infty &\text{otherwise,} \end{cases}
\end{equation}
where $p : \tilde{M} \rightarrow M$ is the universal covering map and $\|\cdot\|$ denotes the dual norm induced by $\g$. One can immediately see that $s_0$ is finite if and only if $p^* \sigma$ admits a bounded primitive, and $s_0 = 0$ if and only if $\sigma = 0$. The Ma\~{n}\'{e} critical value marks a drastic change of the dynamics of $\varPhi^{(\g, \sigma)}_s$, and its computation is challenging in general. For more details and alternative definitions of $s_0$, we point the reader to \cite{Merry_2010,CieliebakFrauenfelderPaternain_2010}, and references within. In the special case when $(M, \g)$ is a K\"{a}hler manifold with K\"{a}hler form $\omega$ and constant holomorphic sectional curvature $\text{sec}^{\text{hol}} \equiv k < 0$, we have
\begin{equation} \label{eqn:mane2}
	s_0(\g, \lambda \omega) = \frac{|\lambda|}{\sqrt{-k}};
\end{equation}
see Appendix \ref{appendixA}.

We say that a magnetic system $(\g, \sigma)$ is \emph{magnetically flat at level $s > 0$} if $M^\Omega_s \equiv 0$. Just like in the surface case, it turns out that being magnetically flat is a rigid condition, and there are only two types of magnetic systems $(\g,\sigma)$ and levels $s$ for which it is satisfied: when $\sigma = 0$ and $\g$ is flat (in which case $s > 0$ can be anything), or when $(\g, \lambda\omega)$ is a K\"{a}hler magnetic system on a K\"{a}hler manifold with constant negative holomorphic sectional curvature and $s=s_0(\g, \lambda \omega)$. This is a consequence of the following theorem,
which we prove in Section \ref{sec:mag_flatness}:


\begin{teoA} \label{thm:mag_flat}
	Let $(\g, \sigma)$ be a magnetic system on a (not necessarily closed) manifold $M$ with $\sigma \neq 0$. If ${\rm sec}^\Omega_s \equiv c$ for some $s > 0$, then $\nabla \sigma = 0$, $\sigma$ is nowhere vanishing, and one of the following conclusions must hold:
	\begin{enumerate}[\normalfont (i)]
		\item \label{thm:mag_flat3} $(M,\g)$ is an oriented surface whose Gaussian curvature is constant and equal to \newline${(c-  \|\Omega\|^2)/s^2}$, and $\sigma = - \|\Omega\| \nu_g$;
		\item $\dim M \geq 4$, $c = 0$, and $\|\Omega\|^{-1} \Omega$ makes $(M, \g)$ a K\"{a}hler manifold with constant holomorphic sectional curvature equal to $-\|\Omega\|^2/s^2$.
		
	\end{enumerate}
	Moreover, assuming that $c=0$ in {\rm(i)}, we have that $s = s_0(\g, \sigma)$ in either case.
\end{teoA}


We conclude by connecting the above discussion with the problem of existence of closed $(\g, \sigma)$-geodesics with speed $s$.
The literature around this problem is vast; for more details, see \cite{Contreras_2006, Merry_2010, Asselle_Benedetti_2016}.
In general, existence depends on whether $s$ is above or below $s_0$. Indeed, for every $s>s_0$, there exists a closed $(\g, \sigma)$-geodesic with speed $s$ \cite[Theorem 27]{Contreras_Iturriaga_Paternain_2000}. Moreover, if $M$ is not simply connected and $s > s_0$, then for every non-trivial free homotopy class contains a $(\g, \sigma)$-geodesic with speed $s$ \cite{Asselle_Benedetti_2016}. On the other hand, for almost every sufficiently small $s<s_0$, there exists a contractible closed $(\g, \sigma)$-geodesic with speed $s$ \cite[Corollary 3.5]{Schlenk_2006}. It is conjectured that the latter conclusion in fact holds for all $s \in (0,s_0)$. To our best knowledge, the only examples of magnetic flows without closed $(\g, \sigma)$-geodesics are the ones of the form $\varPhi^{(\g, \lambda \omega)}_{s_0}$, where $(\g, \lambda \omega)$ is a K\"ahler magnetic system with constant negative holomorphic sectional curvature and $s_0$ is the Ma\~{n}\'{e} critical value.
This raises the following question: \emph{if there are no closed $(\g, \sigma)$-geodesics with speed $s$, is the system $(\g, \sigma)$ magnetically flat at level $s$?}
%

\section*{Acknowledgements and funding}
	We would like to thank Jacopo de Simoi and Andrey Gogolev for their support throughout this project. We would also like to thank Johanna Bimmerman for her notes on closed orbits for K\"{a}hler manifolds. Finally, we would like to thank Gabriele Benedetti, Andrey Gogolev, Leonardo Macarini, Gabriel Paternain, and Jacopo de Simoi for feedback on early drafts of the paper.

All authors were supported by NSERC Discovery Grant RGPIN-2022-04188. The first author was support by Deutsche Forschungsgemeinschaft (DFG, German Research Foundation) – Project-ID 281071066 –- TRR 191 and by the Instituto Serrapilheira. The second and third authors were supported by NSF DMS-2247747. The third author was also supported by FAPESP--OSU Regular Research Award 2015/50265-6.

	\section{The anisotropic objects $\tOmega$ and $\tD$} \label{sec:anisotropic}

In this section, we fix a magnetic system $(\g,\sigma)$ on a (not necessarily closed) manifold $M$, and denote by $\pi\colon TM\to M$ the natural projection. We construct the anisotropic Lorentz force and the anisotropic twisted connection, which will be useful tools for the calculations done in Sections \ref{sec:curvature-identity}, \ref{sec:mag_green_bundles}, and \ref{sec:proof_thm_a}.

\subsection{The anisotropic Lorentz force} The \emph{anisotropic Lorentz force} $\tOmega$ is defined by assigning to every $v\in SM$ the skew-adjoint operator $\tOmega_v\colon T_{\pi(v)}M \to T_{\pi(v)}M$ given by
\begin{equation}\label{eqn:tOmega-def}
    \widetilde{\Omega}_v(w) = \Omega(w)^\top + \Omega(w^\top)+ \frac{1}{2} \Omega(w^\perp)^\perp,
\end{equation}
where $[\,\cdot\,]^\top$ and $[\,\cdot\,]^\perp$ denote the projections onto $\R v$ and $v^\perp$, according to the direct sum decomposition $T_{\pi(v)}M = \R v\oplus v^\perp$. A straightforward computation yields the alternative formula
\begin{equation}\label{eqn:tOmega_alternative}
    \widetilde{\Omega}_v(w) = \frac{1}{2}\left(\langle  \Omega(w),v\rangle v + \langle w,v\rangle \Omega(v) + \Omega(w)\right)
\end{equation}
for $\tOmega$. In particular, note that $\tOmega_{v}$ is completely characterized by the relations
\begin{equation}\label{eqn:prop_tOmega}
  {\rm i})~\tOmega_v(v) = \Omega(v)\quad {\rm ii})~\langle\tOmega_v(w),v\rangle = \langle \Omega(w),v\rangle\quad{\rm iii})~\tOmega_v(w^\perp)^\perp = \dfrac{1}{2}\Omega(w^\perp)^\perp
\end{equation}
imposed on all $w\in T_{\pi(v)}M$. Repeatedly applying \eqref{eqn:tOmega_alternative} and recalling \eqref{eqn:AOmega} yields
\begin{equation}\label{eqn:tOmega2}
    \tOmega_v^2(w) =-(A^\Omega)_v(w^\perp) + \frac{1}{2} (\Omega^2(w^\top) + \Omega^2(w)^\top).
    \end{equation}

 \subsection{The anisotropic twisted connection}   

It is well-known that, if $\sigma \neq 0$, there is no affine connection on $M$ whose geodesics are the $(\g,\sigma)$-geodesics of our fixed magnetic system \cite[Proposition 2.1]{GLH_2005}. Thus, given a $(\g,\sigma)$-geodesic $\gamma\colon I\to M$ with constant speed $s>0$, we may search for a covariant derivative operator $\tD$, defined only along $\gamma$, which describes geometric properties of $(\g,\sigma)$ in the same way that the Levi-Civita covariant derivative ${\rm D}/{\rm d}t$ does so for the metric $\g$.

The first two desired properties are (i) $\tD \dot{\gamma}=0$, and (ii) metric compatibility. Condition (i) should ultimately be equivalent to \eqref{eqn:Landau-Hall} and, in the presence of (ii), implies that $\tD$-derivatives of vector fields proportional to $\dot{\gamma}$ remain proportional to $\dot{\gamma}$. Condition (ii) also implies that $\tD$-derivatives of vector fields orthogonal to $\dot{\gamma}$ remain orthogonal to $\dot{\gamma}$. This allows us to formally deduce, from \eqref{eqn:magnetic_Jacobi_equation}, that $\tD^2J^\perp + A\tD J^\perp + BJ^\perp = 0$ for every $(\g,\sigma)$-Jacobi field $J$ along $\gamma$, for suitable endomorphisms $A$ and $B$ of the pullback bundle $\gamma^*E$. The last desired condition, inspired by the Riemannian situation, is (iii) $A = 0$ and $B = (M^\Omega_s)_{\dot{\gamma}/s}$.

As we will see in Proposition \ref{propn:curvature-identity}, the covariant derivative operator $\tD$ acting on vector fields $V$ along $\gamma$ by
 \begin{equation}\label{eqn:Dtilde}
   \tD V = \frac{{\rm D}V}{{\rm d}t} - \tOmega_{\dot{\gamma}/s}(V),
 \end{equation}for $\tOmega$ is as in \eqref{eqn:tOmega-def}, satisfies all the requirements.

\smallskip
 
\noindent {\bf Warning:} For economy of notation, we will write simply $\tOmega$ instead of $\tOmega_{\dot{\gamma}/s}$ when using \eqref{eqn:Dtilde} in the next sections.
 
 \section{Curvature identities} \label{sec:curvature-identity}

The main goal of this section is to establish the following result:
 
 \begin{prop}\label{propn:curvature-identity}
   Let $(\g,\sigma)$ be a magnetic system on a (not necessarily closed) manifold $M$, $\gamma\colon I\to M$ be a $(\g,\sigma)$-geodesic with speed $s>0$, and $J$ be a normal $(\g,\sigma)$-Jacobi field along $\gamma$. Then
   \begin{equation} \label{eqn:Jacobi}
     \tD^2 J^\perp + (M^\Omega_s)_{\dot{\gamma}/s}(J^\perp) = 0
   \end{equation}on $I$.
 \end{prop}

 Here, we will need to consider the projections ${\rm P}^\top$ and ${\rm P}^\perp$, acting on vector fields $V$ along $\gamma$ via 
\begin{equation*}
	{\rm P}^\top(V) = V^\top =  s^{-2}\langle V,\dot{\gamma}\rangle \dot{\gamma}\quad\mbox{and}\quad {\rm P}^\perp(V) = V^\perp = V - s^{-2}\langle V,\dot{\gamma}\rangle \dot{\gamma}.
\end{equation*}
Denoting Levi-Civita covariant derivatives with dots, a straightforward computation shows that $({\rm P}^\top)^{\cdot} = [\Omega, {\rm P}^\top]$ and $({\rm P}^\perp)^{\cdot} = [\Omega, {\rm P}^\perp]$. With these commutator identities in place, it follows from \eqref{eqn:tOmega-def} and from differentiating $\tOmega(\dot{\gamma}) = \Omega(\dot{\gamma})$ that
 \begin{equation}\label{eqn:tOmega-dot}
  {\rm i})~ {\rm P}^\perp \circ \dot{\tOmega}\circ {\rm P}^\perp = \frac{1}{2}    {\rm P}^\perp \circ \dot{\Omega}\circ {\rm P}^\perp \quad\mbox{and}\quad {\rm ii})~\dot{\tOmega}(\dot{\gamma}) = \dot{\Omega}(\dot{\gamma}) + \frac{1}{2} \Omega^2(\dot{\gamma})^\perp,
 \end{equation}respectively.

 \begin{proof}[Proof of Proposition \ref{propn:curvature-identity}]
 	We start by computing
 	\begin{equation}\label{eqn:curvature-identity}
 		\tD^2J = \tD(\dot{J}- \tOmega(J)) = \ddot{J} - 2\tOmega(\dot{J})-\dot{\tOmega}(J) + \tOmega^2(J),
 	\end{equation}
 	and substituting the $(\g,\sigma)$-Jacobi equation \eqref{eqn:magnetic_Jacobi_equation} together with \eqref{eqn:tOmega2} to obtain
 	\begin{equation}\label{eqn:D2J}
 		\begin{split}
 			\tD^2J &= R(\dot{\gamma},J)\dot{\gamma} + (\nabla_J\Omega)(\dot{\gamma}) + \Omega(\dot{J}) - 2\tOmega(\dot{J}) \\ &\quad - \dot{\tOmega}(J) -(A^\Omega)_{\dot{\gamma}/s}(J^\perp) + \frac{1}{2}(\Omega^2(J^\top) + \Omega^2(J)^\top).
 		\end{split}
 	\end{equation}
 	At the same time, using that (\ref{eqn:tOmega-dot}-ii) remains valid when replacing $\dot{\gamma}$ with $J^\top$, we may evaluate \eqref{eqn:ROmega} as
 	\begin{equation}\label{eqn:RsOm}
 		(R^\Omega_s)_{\dot{\gamma}/s}(J^\perp) = -R(\dot{\gamma},J^\perp)\dot{\gamma} - (\nabla_{J^\perp}\Omega)(\dot{\gamma}) + \dot{\tOmega}(J^\perp)^\perp.
 	\end{equation}
 	As $R(\dot{\gamma},J)\dot{\gamma} = R(\dot{\gamma},J^\perp)\dot{\gamma}$ and $(\nabla_J\Omega)(\dot{\gamma}) = \dot{\Omega}(J^\top) + (\nabla_{J^\perp}\Omega)(\dot{\gamma})$, we may substitute \eqref{eqn:RsOm} into \eqref{eqn:D2J} and use \eqref{def:MOmega} to write
 	\begin{equation}\label{eqn:D2J-afterM}
 		\begin{split}
 			\tD^2J &= -(M^\Omega_s)_{\dot{\gamma}/s}(J^\perp)+  \dot{\tOmega}(J^\perp)^\perp+  \dot{\Omega}(J^\top)+\Omega(\dot{J}) \\ &\qquad\qquad -2\tOmega(\dot{J}) - \dot{\tOmega}(J) + \frac{1}{2}(\Omega^2(J^\top) + \Omega^2(J)^\top).
 		\end{split}
 	\end{equation}
 	Applying ${\rm P}^\perp$ to both sides of \eqref{eqn:D2J-afterM} and noting that $\Omega(\dot{J})^\perp - 2\tOmega(\dot{J})^\perp = 0$, by normality of $J$ and \hbox{(\ref{eqn:prop_tOmega}-iii)}, it follows that
 	\begin{equation}\label{eqn:D2J-afterM-2}
 		\tD^2J^\perp = -(M^\Omega_s)_{\dot{\gamma}/s}(J^\perp) +\dot{\tOmega}(J^\perp)^\perp + \dot{\Omega}(J^\top)  - \dot{\tOmega}(J)^\perp + \frac{1}{2}\Omega^2(J^\top)^\perp .
 	\end{equation}
 	Decomposing $\dot{\tOmega}(J)^\perp = \dot{\tOmega}(J^\top)^\perp + \dot{\tOmega}(J^\perp)^\perp$ and using \hbox{(\ref{eqn:tOmega-dot}-ii)} for a second time, a complete cancellation occurs and \eqref{eqn:D2J-afterM-2} reduces to \eqref{eqn:Jacobi}, as required.
 \end{proof}
 
 We conclude this section with one interesting consequence of Proposition \ref{propn:curvature-identity}.

 \begin{prop}\label{prop:negative_sec_no_conj}
   Let $(\g,\sigma)$ be a magnetic system on a (not necessarily closed) manifold $M$. If $s>0$ is such that ${\rm sec}^\Omega_s \leq 0$, then $\varPhi^{(\g,\sigma)}_s$ is without conjugate points.
 \end{prop}

 \begin{proof}
   Let $v\in \Sigma_s$, and assume that $J$ is any normal $(\g,\sigma)$-Jacobi field along the geodesic $\gamma$ with initial velocity $v$. Using metric compatibility of $\tD$ together with \eqref{eqn:Jacobi}, we find that
   \begin{equation}
     \begin{split}
       \frac{{\rm d}^2}{{\rm d}t^2} \frac{\|J^\perp\|^2}{2} &= \|\tD J^\perp\|^2 + \langle \tD^2J^\perp, J^\perp\rangle \\ &=  \|\tD J^\perp\|^2 - \langle (M^\Omega_s)_{\dot{\gamma}/s}(J^\perp), J^\perp\rangle \\ &= \|\tD J^\perp\|^2 - ({\rm sec}^\Omega_s)_{\dot{\gamma}/s}\left(\frac{J^\perp}{\|J^\perp|}\right) \|J^\perp\|^2 \geq 0,
     \end{split}
   \end{equation}where the ${\rm sec}^\Omega_s$ term is understood to be zero at all instants on which $J^\perp$ vanishes. Hence, the function $t\mapsto \|J(t)^\perp\|$ is increasing whenever $J$ is normal. If, in addition, we suppose that $J(0) = 0$ and that $J(t_0)$ is proportional to $\dot{\gamma}(t_0)$ for some $t_0>0$, then it must be the case that $J^\perp|_{[0,t_0]} = 0$ identically, making $J|_{[0,t_0]}$ tangential. As a tangential and normal $(\g,\sigma)$-Jacobi field is necessarily a constant multiple of $\dot{\gamma}$, $J(0)=0$ implies that $J|_{[0,t_0]}=0$, and so $\gamma(t_0)$ is not conjugate to $\gamma(0)$ along $\gamma$.
 \end{proof}
 
\section{Magnetic Green bundles}\label{sec:mag_green_bundles}

In this section, we construct the so-called Green bundles associated with a magnetic flow $\varPhi^{(\g, \sigma)}_s$. As described in the Introduction, this is a useful tool for obtaining non-trivial normal $(\g, \sigma)$-Jacobi fields along any given $(\g,\sigma)$-geodesic whose perpendicular component never vanishes. The strategy is to introduce a symplectic vector bundle $Q$ over $\Sigma_s$, whose fibers are in correspondence with suitable spaces of normal $(\g, \sigma)$-Jacobi fields (see Lemma \ref{lem:quotient_correspondence}).  The Green bundles are then defined as limits of curves of Lagrangian subbundles of Q (see Lemma \ref{lem:lim_Sv_exist}), and its elements give rise to the desired fields. 



\subsection{General setup} \label{sec:setup}

Let $(\g,\sigma)$ be a magnetic system on a closed manifold $M$. On $TM$, we may consider the \emph{twisted symplectic form} $\omega^\sigma$ defined by
\begin{equation}\label{eqn:defn_twisted_symp}
    \omega^\sigma(\xi,\eta) = \langle \pi_\ast\xi, K(\eta)\rangle - \langle K(\xi), \pi_\ast\eta\rangle - \sigma(\pi_\ast\xi,\pi_\ast\eta),
\end{equation}
for all $\xi,\eta\in T(TM)$. Here, $\pi\colon TM\to M$ is the natural bundle projection, and $K\colon TTM\to TM$ is the \emph{Levi-Civita connector} of $(M,\g)$. 

It is well-known that 
\begin{equation}\label{eqn:mag-ham}
    \parbox{.73\textwidth}{the magnetic flow $\varPhi^{(\g,\sigma)} \colon \mathbb{R} \times TM\to TM$ is Hamiltonian,} 
\end{equation}
as its generator $X^{\g,\sigma} \in \mathfrak{X}(TM)$ satisfies $\omega^\sigma(X^{\g,\sigma}, \cdot) = {\rm d}L$, where $L\colon TM\to \R$ is given by $L(v) = \|v\|^2/2$. Indeed, it is a direct consequence of \eqref{eqn:defn_twisted_symp} together with 
\begin{equation}\label{eqn:KXsigma}
    {\rm (i)}~ {\rm d}L_v(\eta) = \langle K(\eta), v\rangle \quad\mbox{and}\quad {\rm (ii)}~ K(X^{\g,\sigma}_v) = \Omega(v).
\end{equation}
Clearly (\ref{eqn:KXsigma}-i) holds, while (\ref{eqn:KXsigma}-ii) follows from the easily verified general relation $X^{\g,\sigma}_v = X^\g_v + \Omega^{\sf v}(v)$, where $X^\g \in\mathfrak{X}(TM)$ denotes the geodesic vector field of $(M,\g)$ and $\Omega^{\sf v}(v) \in T_{v}(TM)$ is the vertical lift of $\Omega(v)$, characterized by $\pi_\ast (\Omega^{\sf v}(v)) = 0$ and $K(\Omega^{\sf v}(v)) = \Omega(v)$.

In the following, instead of $K$ we will use 
\begin{equation}\label{eqn:defn_Ksigma}
    \parbox{.63\textwidth}{the \emph{twisted connector} $K^\sigma\colon TTM \smallsetminus \{0\}\to TM$, defined by $K^\sigma(\xi) = K(\xi) - \tOmega(\pi_\ast\xi)$, for $\xi\in T_v(TM)$,}
\end{equation}
where $\tOmega$ is the anisotropic Lorentz force given by \eqref{eqn:tOmega-def}. Note that the zero section must be removed from $TTM$ in \eqref{eqn:defn_Ksigma} so that $\tOmega$ may be evaluated. In addition, observe that $K^\sigma(X^{\g,\sigma})= 0$ as a consequence of (\ref{eqn:KXsigma}-ii) and (\ref{eqn:prop_tOmega}-i).

\subsection{Reduced magnetic Green bundles over $s$-sphere bundles} \label{sec:reduced_magnetic_green_bundles}

Let $s>0$ and fix $\Sigma_s = \{v\in TM \mid \|v\|=s\}$. We denote by $\pi$ and $K^\sigma$ the restrictions to $\Sigma_s$ of the corresponding objects in $TM$, and by $\varPhi^{(\g, \sigma)}_s$ the magnetic flow restricted to $\Sigma_s$.

Combining (\ref{eqn:KXsigma}-i) with \eqref{eqn:defn_Ksigma} and (\ref{eqn:prop_tOmega}-i), we see that 
\begin{equation}\label{eqn:tan_space_shell}
    T_v\Sigma_s = \{\xi \in T_v(TM) \mid \langle K^\sigma(\xi)+\Omega(\pi_\ast\xi),v\rangle = 0\},
\end{equation}for every $v\in \Sigma_s$. Three distinguished subspaces of \eqref{eqn:tan_space_shell} are the \hbox{\emph{$(\g,\sigma)$-horizontal}} and \emph{vertical} spaces at $v$, defined by $H^\sigma_v = \ker K^\sigma|_{T_v\Sigma_s}$ and $V_{v} = \ker \pi_\ast|_{T_v\Sigma_s}$, respectively, and $\Theta_v = \{  \xi \in T_v(TM) \mid \pi_\ast \xi \in \R \Omega(v) \text{ and } K(\xi) = 0 \}$. They provide 
\begin{equation}\label{eqn:decomp_HV}
    \parbox{.65\textwidth}{a direct sum decomposition $T_v\Sigma_s =  H^\sigma_v\oplus \Theta_v \oplus V_v.$}
\end{equation}Indeed, it is straightforward to check that $H_{v}^\sigma \cap V_{v} = \{0\}$ by using that $K$ and $K^\sigma$ agree on vertical vectors, as well as $\Theta_v\cap H^\sigma_v = \Theta_v \cap V_v = \{0\}$ using \eqref{eqn:tan_space_shell}, while the restrictions  \begin{equation}\label{eqn:restr_K_pi}
    {\rm (i)}\,K^\sigma\colon V_v \to v^\perp\quad\mbox{and}\quad {\rm (ii)}\,
\pi_\ast \colon H^\sigma_v \oplus \Theta_v \to T_{\pi(v)}M\end{equation}are isomorphisms, so that $\dim T_v\Sigma_s = \dim (H^\sigma_v \oplus \Theta_v) + \dim V_v$. Motivated by \mbox{(\ref{eqn:restr_K_pi}-ii)}, let $\hat{H}^\sigma_v = H^\sigma_v \oplus \Theta_v$. Formula \eqref{eqn:defn_twisted_symp} can now be made simpler:

\begin{lem}\label{lem:simplified_Ksigma}
   With the above setup,
\begin{equation}\label{eqn:simplified_Ksigma} \omega^\sigma(\xi,\eta) = \langle [\pi_\ast\xi]^\perp, K^\sigma(\eta)\rangle - \langle K^\sigma(\xi),[\pi_\ast\eta]^\perp\rangle,  \end{equation}for all $\xi,\eta\in T\Sigma_s$. In particular, $\omega^\sigma$ restricted to $\Sigma_s$ is degenerate and has its kernel spanned by $X^{\g,\sigma}$.
\end{lem}

\begin{proof}
   For $\xi,\eta\in T_v\Sigma_s$, decompose $T_{\pi(v)}M = \R v \oplus v^\perp$, and denote by $[\,\cdot\,]^\top$ and $[\,\cdot\,]^\perp$ the orthogonal projections onto $\R v$ and $v^\perp$, respectively. Using \eqref{eqn:tOmega_alternative} and \eqref{eqn:Lorentz_force_def}, \eqref{eqn:defn_twisted_symp} becomes
    \begin{equation}
    \begin{split}
         \omega^\sigma(\xi,\eta) &= \langle \pi_\ast\xi, K^\sigma(\eta)\rangle - \langle K^\sigma(\xi), \pi_\ast\eta\rangle \\ &\quad +\langle \Omega(\pi_\ast\eta), [\pi_\ast\xi]^\top\rangle-\langle \Omega(\pi_\ast\xi), [\pi_\ast\eta]^\top\rangle.
    \end{split}
  \end{equation}As $\xi,\eta\in T_v\Sigma_s$, self-adjointness of $[\,\cdot\,]^\top$ and $[\,\cdot\,]^\perp$ together with \eqref{eqn:tan_space_shell} leads to \eqref{eqn:simplified_Ksigma}. It remains to establish the claim about the kernel of $\omega^\sigma$. If $\xi \in T_v\Sigma_s$ is such that $\omega^\sigma(\xi,\cdot) = 0$, then from the relation $\omega^\sigma(\xi , V_v) = 0$ and (\ref{eqn:restr_K_pi}-i) we directly obtain that $\pi_\ast\xi \in (v^\perp)^\perp = \R v$. Writing $\pi_\ast\xi = \lambda v$, for some $\lambda \in \R$, and noting from \mbox{(\ref{eqn:restr_K_pi}-ii)} that $H^\sigma_v$ surjects onto $\Omega(v)^\perp$, we see that $\omega^\sigma(\xi,H^\sigma_v) = 0$ implies that $K^\sigma(\xi) \in \R \Omega(v)$, since $K^\sigma(\xi)$ lies in the double orthogonal complement to $\R\Omega(v)$ \emph{relative to $v^\perp$}. We conclude with two cases: if $\Omega(v) = 0$, then $K(\xi) = 0$ and thus $\xi = \lambda X^\g_v = \lambda X^{\g,\sigma}_v$; if $\Omega(v) \neq 0$ instead, then $K(\xi) - \lambda\Omega(v) = \mu \Omega(v)$ for some $\mu \in \R$, and we use that $\omega^{\sigma}(\xi,\eta) = -\mu \|\Omega(v)\|^2$ whenever $\eta\in \Theta_v$ has $\pi_\ast\eta = \Omega(v)$, so that $\mu = 0$ and $\xi = \lambda X^{\g,\sigma}_v$ again.
\end{proof}

By modding out $\R X^{\g,\sigma}$ in \eqref{eqn:decomp_HV}, we obtain a symplectic vector bundle $Q$ over $\Sigma_s$, whose fiber over $v\in \Sigma_s$ is given by $Q_v = T_v\Sigma_s/\R X^{\g,\sigma}_v$ and whose symplectic structure is naturally induced by $\omega^\sigma$. As $X^{\g,\sigma}_v \not\in V_v$, we have that $V$ is projectable, and so we may also identify $Q_v$ with the direct sum $[\hat{H}^\sigma_v/\R X^{\g,\sigma}_v]\oplus V_v$. In addition, due to \eqref{eqn:mag-ham} and $X^{\g,\sigma}$ being \hbox{$\varPhi^{(\g,\sigma)}_s$-related} to itself, we see that
\begin{equation}\label{eqn:induces_symplec}
  \parbox{.85\textwidth}{for each $t\in \R$, the quotient flow $\varPsi^{\g,\sigma} \colon \mathbb{R} \times Q\to Q$ induced by the derivative $\,{\rm d}\varPhi^{(\g,\sigma)}_s\colon \mathbb{R}\times T\Sigma_s\to T\Sigma_s\,$ is a bundle symplectomorphism.}
\end{equation}

For notational simplicity, we write $\varPsi^{(\g, \sigma)}_t = \varPsi^{(\g, \sigma)}(t, \cdot)$. For every $v\in \Sigma_s$ and $t\in \R$, we define $E^\sigma_v(t) = (\varPsi^{(\g,\sigma)}_t)^{-1}[V_{\varPhi^{(\g,\sigma)}_s(t,v)}]$, so that $(\varPsi^{(\g,\sigma)}_t)_v[E^\sigma_v(t)] = V_{\varPhi^{(\g,\sigma)}_s(t,v)}$. Here, we identify each $V_v$ with its isomorphic image under the quotient projection $T\Sigma_s\to Q$. It follows from \eqref{eqn:simplified_Ksigma} that $V$ is a Lagrangian subbundle of $Q$, and thus, by \eqref{eqn:induces_symplec},
\begin{equation}\label{eqn:curve_of_Lagrangians}
    \parbox{.685\textwidth}{$\R\ni t\mapsto E^\sigma_v(t)$ is a curve of Lagrangian subspaces of $Q_v$.}
\end{equation}

To study the behavior of $E^\sigma_v(t)$ as $t\to \pm \infty$, it will be convenient to use a different characterization of $Q$. For each $v \in \Sigma_s$, we let $\mathcal{J}_0^{\g,\sigma, {\rm n}}(v)$ be the space of all normal $(\g,\sigma)$-Jacobi fields along $\gamma_v$ such that $J(0)^\top = 0$.

\begin{lem} \label{lem:quotient_correspondence}
The mapping
\begin{equation}\label{eqn:ident_Jac}
  Q_v \ni \xi + \R X^{\g,\sigma}_v \mapsto J_\xi \in \mathcal{J}_0^{\g,\sigma ,{\rm n}}(v),
\end{equation}where $J_\xi$ is the unique normal $(\g,\sigma)$-Jacobi field determined by the initial conditions $J_\xi(0) = [\pi_*\xi]^\perp$ and $(\tD J_\xi)(0) = K^\sigma(\xi)$, is a well-defined linear isomorphism.   
\end{lem}

\begin{proof}	
It is easy to see that if one replaces $\xi$ with $\xi + \lambda X^{\g,\sigma}_v$ for some $\lambda \in \R$, then such initial conditions remain unchanged, while $J_\xi$ is automatically normal in view of $\xi \in T_v\Sigma_s$ combined with \eqref{eqn:tan_space_shell}. Finally, \eqref{eqn:ident_Jac} is injective for reasons explained in the proof of Lemma \ref{lem:simplified_Ksigma}: namely, if $[\pi_\ast\xi]^\perp = 0$, so that $\pi_\ast\xi=\lambda v$ for some $\lambda \in \R$, then $K^\sigma(\xi) = 0$ implies that $\xi = \lambda X^{\g,\sigma}_v$, and thus $\xi+\R X^{\g,\sigma}_v$ is zero in $Q_v$. As $\dim Q_v = \dim \mathcal{J}_0^{\g,\sigma ,{\rm n}}(v) = 2(n-1)$, 
it follows that \eqref{eqn:ident_Jac} is an isomorphism.
  
\end{proof}

As the next result shows, the condition of $\varPhi^{(\g,\sigma)}_s$ being without conjugate points is what will allows us to proceed.

\begin{prop} \label{prop:intersection}
  Let $\gamma$ be the $(\g,\sigma)$-geodesic with initial velocity $v \in \Sigma_s$. For each $\tau \in \R$, $E^\sigma_v(\tau) \cap V_v \neq \{0\}$ if and only if $\gamma(\tau)$ is conjugate to $\gamma(0) = \pi(v)$ along $\gamma$.
\end{prop}

\begin{proof}
  Note that, under the identification \eqref{eqn:ident_Jac}, we may express the vertical space as $V_v \cong \{ J \in \mathcal{J}^{\g,\sigma,{\rm n}}_0(v) \mid J(0)=0 \}$. The quotient flow introduced in \eqref{eqn:induces_symplec}, in turn, appears as the mapping $\varPsi^{\g,\sigma}_\tau\colon \mathcal{J}_0^{\g,\sigma,{\rm n}}(v) \to \mathcal{J}_0^{\g,\sigma,{\rm n}}(\varPhi^{(\g,\sigma)}_s(\tau,v))$ given by \begin{equation}\label{eqn:quot_flow_J} [\varPsi^{\g,\sigma}_\tau(J)](t) = J(t+\tau) - s^{-2}\langle J(\tau),\dot{\gamma}(\tau)\rangle \dot{\gamma}(t+\tau).\end{equation}
  Indeed, that the derivative of the magnetic flow pushes $J$ to $t\mapsto J(t+\tau)$ is obtained exactly as in \cite[Lemma 1.40]{Paternain99_book}, while the term subtracted in \eqref{eqn:quot_flow_J} amounts to the natural projection of a normal $(\g,\sigma)$-Jacobi field along $\gamma$ onto $\mathcal{J}_0^{\g,\sigma,{\rm n}}(v)$. As $[\varPsi^{\g,\sigma}_\tau(J)](0) = J(\tau)^\perp$, we see that $
    E^\sigma_v(\tau) \cong \{J \in \mathcal{J}^{\g,\sigma,{\rm n}}_0(v) \mid J(\tau)^\perp =0\}$
under \eqref{eqn:ident_Jac}, so that 
\begin{equation} \label{eqn:E-sigma-tau}
E^\sigma_v(\tau) \cap V_v \cong \{J \in \mathcal{J}^{\g,\sigma,{\rm n}}_0(v) \mid J(0) = J(\tau)^\perp =0\}.
\end{equation} 
The conclusion follows.
\end{proof}


From here on, we assume that $\varPhi^{(\g,\sigma)}_s$ is without conjugate points. Observe the isomorphism \hbox{$\hat{H}^\sigma_v/\R X^{\g,\sigma}_v \cong v^\perp$} induced by (\ref{eqn:restr_K_pi}-ii), in view of $\pi_\ast (X^{\g,\sigma}_v) = v$ and $T_{\pi(v)}M/\R v \cong v^\perp$. With this isomorphism and (\ref{eqn:restr_K_pi}-i) in mind, we have
\begin{equation}\label{eqn:E-graph-S}
  \parbox{.87\textwidth}{each $E^\sigma_v(t)$ corresponds to the graph of a linear operator $S_v(t)\colon v^\perp \to v^\perp$}
\end{equation}
which, by \eqref{eqn:simplified_Ksigma} and \eqref{eqn:curve_of_Lagrangians}, must be self-adjoint.
Now, recall the \emph{Loewner order} on the space of self-adjoint operators on $v^\perp$: we write $A<B$ if $B-A$ is positive-definite. For this partial order, it is well-known that
\begin{equation}\label{mct}
  \parbox{.78\textwidth}{the \emph{monotone convergence theorem} holds: a bounded and mo\-no\-to\-ne sequence (or curve) of self-adjoint operators must converge,}
\end{equation}
cf. \cite[3.6.5, p. 108]{Bobrowski}.

\begin{lem}\label{lem:lim_Sv_exist}
For every $v \in \Sigma_s$, the limits
  \begin{equation}
  S_v^+ = \lim_{t\to +\infty} S_v(t)\quad\mbox{and}\quad S_v^- = \lim_{t\to -\infty} S_v(t)  
  \end{equation}
   exist, and their graphs define $\varPsi^{(\g, \sigma)}_t$-invariant Lagrangian subbundles \hbox{$E^{\sigma,+}$ and $E^{\sigma,-}$ of $Q$}.
\end{lem}

\begin{proof}
The arguments for $S_v^+$ and $S_v^-$ are identical, so it suffices to show that $S_v^+$ exists and its graph defines an invariant Lagrangian subbundle. We first show that the limit exists, following the argument in \cite{Iturriaga92}. Assume for simplicity that $s=1$, so $\Sigma_s=SM$. We claim that
  \begin{equation}\label{eqn:bounded_monotone}
    \parbox{.69\textwidth}{the difference $S_v(t,s) = S_v(t)-S_v(s)$ has constant signature in the open regions (i) $0<s<t$ and (ii) $s<0<t$.}
  \end{equation}
  To verify \hbox{(\ref{eqn:bounded_monotone}-i)}, assume that there exists $w\in \ker S_v(t_0,s_0) \smallsetminus \{0\}$, for suitable $0<s_0<t_0$. Then, the field $J \in \mathcal{J}^{\g,\sigma,{\rm n}}_0(v)$ with initial conditions $J(0) = w$ and $(\tD J)(0) = S_v(t_0)w$ necessarily has $J(t_0)^\perp = J(s_0)^\perp = 0$, due to \eqref{eqn:E-sigma-tau} and \eqref{eqn:E-graph-S}. It follows from \eqref{eqn:quot_flow_J} that $\Psi^{(\g, \sigma)}_{s_0}(J) \in \mathcal{J}^{\g,\sigma,\rm{n}}_0(\varPhi^{(\g,\sigma)}_1(s_0,v))$ satisfies $\Psi^{(\g, \sigma)}_{s_0}(J)(0) = 0$ and 
  $[\Psi^{(\g, \sigma)}_{s_0}(J)(t_0-s_0)]^\perp = 0$, contradicting our assumption that $\varPhi^{(\g,\sigma)}_1$ is without conjugate points. The same argument also establishes \hbox{(\ref{eqn:bounded_monotone}-ii)}.

  In view of (\ref{eqn:bounded_monotone}-i), to show that $t\mapsto S_v(t)$ is increasing for $t>0$, it suffices to exhibit $0<s_0<t_0$ such that $S_v(t_0,s_0)$ is positive-definite. This is done with a local argument: for every $t>0$, let $P_{-t}\colon \dot{\gamma}(t)^\perp \to v^\perp$ be the $\tD$-parallel transport along $\gamma_v$ and, for every $w\in v^\perp$ let $J \in E^\sigma_v(t)$ have initial conditions $J(0)=w$ and $(\tD J)(0)=S_v(t)w$ --- here, we use \eqref{eqn:ident_Jac}. Now the Taylor expansion
  \begin{equation}
    \begin{split}
    0 &= P_{-t}(J(t)^\perp) = J(0)^\perp + t (\tD J^\perp)(0) + {\rm O}(t^2) \\ &= w + t S_v(t)w + {\rm O}(t^2) = ({\rm Id}_{v^\perp} + tS_v(t) + {\rm O}(t^2))w     
    \end{split}
  \end{equation}leads to $S_v(t) = -t^{-1} {\rm Id}_{v^\perp} + {\rm O}(t)$, meaning that $S_v(t,s)$ is positive-definite for $0<s<t$ sufficiently small.

  Finally, it follows from (\ref{eqn:bounded_monotone}-ii) that $t\mapsto S_v(t)$ is bounded for $t>0$ as, for instance, $S_v(t)  < S_v(-1)$. The existence of $\lim_{t\to +\infty} S_v(t)$ is now a consequence of \eqref{mct}.
  
  To prove invariance of the bundle, fix $\tau \neq 0$ and $v \in SM$. Observe that if $J \in \mathcal{J}_0^{\g, \sigma, \rm{n}}(v)$ has initial conditions $J(0) = w$ and $(\tD J)(0) = S_v(t)w$, then $\varPsi^{(\g, \sigma)}_\tau(J) \in \mathcal{J}_0^{\g, \sigma, \rm{n}}(\varPhi^{(\g, \sigma)}_1(\tau, v))$ is a new Jacobi field satisfying $\varPsi^{(\g, \sigma)}_\tau(J)(0) = J(\tau)^\perp$ and $[\varPsi^{(\g, \sigma)}_\tau(J)(t-\tau)]^\perp = 0$. This shows  \hbox{$
  	(\tD \varPsi^{(\g, \sigma)}_\tau(J))(0) = S_{\varPhi^{\g, \sigma}_1(\tau, v)}(t-\tau)J(\tau)^\perp.$} We let $t$ tend to $+ \infty$ to get that $\varPsi^{(\g, \sigma)}_\tau(E^{\sigma, +}_v) \subseteq E^{\sigma, +}_{\varPhi^{(\g, \sigma)}_1(\tau, v)}$. The choice of $v$ and $\tau$ was arbitrary, so the same argument also shows that  $\varPsi^{(\g, \sigma)}_{-\tau}\big(E^{\sigma, +}_{\varPhi^{\g, \sigma}_1(\tau, v)}\big) \subseteq E^{\sigma, +}_v$, and we conclude invariance. Finally, we observe that the bundles are Lagrangian by Lemma \ref{lem:simplified_Ksigma}, since a limit of symmetric matrices is symmetric.
  
\end{proof}

\section{Proof of Theorem \ref{teo:magnetic-Hopf}} \label{sec:proof_thm_a}

Let $(\g,\sigma)$ be a magnetic system on a closed manifold $M$, and $s>0$ be such that $\varPhi^{(\g,\sigma)}_s \colon \R \times \Sigma_s\to\Sigma_s$ is without conjugate points. We use the notation and setup from Section \ref{sec:mag_green_bundles}. For every $v\in \Sigma_s$ and $t\in \R$, we define an operator $Y_v(t)\colon v^\perp \to v^\perp$ by $Y_v(t)w = P_{-t} (J_w(t)^\perp)$, where $P_{-t}\colon \dot{\gamma}(t)^\perp \to v^\perp$ is the $\tD$-parallel transport operator along $\gamma_v$, and $J_w \in \mathcal{J}^{\g,\sigma,\rm{n}}_0(v)$ has the initial conditions $J_w(0) = w$ and $(\tD J_w)(0)= S^+_vw$, for $S^+_v$ given as in Lemma \ref{lem:lim_Sv_exist}. We first claim that
\begin{equation}\label{eqn:Psi-cocycle}
  P_{-\tau} \circ Y_{\varPhi^{(g,\sigma)}_s(\tau, v)}(t) \circ P_\tau \circ Y_v(\tau) = Y_v(t+\tau)
\end{equation}holds for all $t,\tau \in \R$. Indeed, if we denote by $L_v(t)\colon v^\perp \to \dot{\gamma}(t)^\perp$ the linear mapping taking $w$ to $J_w(t)^\perp$, the cocycle relation $L_{\varPhi^{(\g,\sigma)}_s(\tau,v)}(t)\circ L_v(\tau) = L_v(t+\tau)$ obviously holds; substituting $L_v(t) = P_t\circ Y_v(t)$ in it yields \eqref{eqn:Psi-cocycle}.

By construction of $S_v^+$, we have that $J_w$ is nowhere vanishing whenever \hbox{$w \neq 0$}, and thus each $Y_v(t)$ is invertible. This allows us to consider the self-adjoint operator $U_v(t) = \dot{Y}_v(t) \circ Y_v(t)^{-1}$ on $v^\perp$. Using Proposition \ref{propn:curvature-identity}, we have that
\begin{equation} \label{eqn:Jacobi2}
 \ddot{Y}_v(t)w  = P_{-t} (\tD^2(J_w)(t)^\perp) = -P_{-t} \left( (M^\Omega_s)_{\varPhi^{(\g,\sigma)}_s(t,v)/s} (J_w(t)^\perp)\right)
\end{equation}
or, equivalently, $\ddot{Y}_v(t) = -P_{-t} \circ (M^\Omega_s)_{\varPhi^{(\g,\sigma)}_s(t,v)/s}\circ L_v(t)$. Thus,
\begin{equation}\label{eqn:Riccati-t}
  \begin{split}
   \dot{U}_v(t)  &= \ddot{Y}_v(t) Y_v(t)^{-1} - \dot{Y}_v(t)  Y_v(t)^{-1} \dot{Y}_v(t) Y_v(t)^{-1}  \\ &=  - P_{-t}\circ  (M^\Omega_s)_{\varPhi^{(\g,\sigma)}_s(t,v)/s} \circ P_t - U_v(t)^2.
  \end{split}
\end{equation}
Taking the trace of \eqref{eqn:Riccati-t}, it follows that
\begin{equation}\label{eqn:before-t-0}
  {\rm tr}\big(\dot{U}_v(t)\big) + {\rm tr}\big(U_v(t)^2\big) + {\rm Ric}^\Omega_s(\varPhi^{(\g,\sigma)}_s(t,v)/s) = 0.
\end{equation}Setting $t=0$ and integrating both sides of \eqref{eqn:before-t-0} with respect to the Liouville measure ${\rm d}\mu_{\g,s} = {\rm d}\nu_{\g}\otimes {\rm d}\mathfrak{m}_s$ induced in $\Sigma_s$, where ${\rm d}\nu_{\g}$ is the Riemannian measure of $(M,\g)$ and ${\rm d}\mathfrak{m}_s$ is the Lebesgue measure on a round sphere of radius $s$, we see that
\begin{equation}\label{eqn:almost-there}
  \int_{\Sigma_s} \left({\rm tr}(\dot{U}_v(0)) + {\rm tr}(U_v(0)^2) \right)\,{\rm d}\mu_{\g,s}(v) + s^{n-1} \int_{SM}  {\rm Ric}^\Omega_s(v)\,{\rm d}\mu_{\g, 1}(v) = 0.
\end{equation}
For the integral of ${\rm Ric}^\Omega_s$ above, note that ${\rm d}\mathfrak{m}_s(v) = s^{n-1}{\rm d}\mathfrak{m}(w)$ for $w=v/s$ and apply Fubini's theorem.

Letting $F\colon \Sigma_s\to \R$ be given by $F(v) = {\rm tr}(U_v(0))$, we claim that ${\rm tr}(\dot{U}_v(0)) = X^{\g,\sigma}(F)(v)$. Indeed, in view of the relation
\begin{equation}\label{eqn:conjugation-type}
 U_{\varPhi^{(\g,\sigma)}_s(\tau,v)}(t) = P_\tau\circ U_v(t+\tau) \circ P_{-\tau}, 
\end{equation}
which is a direct consequence of \eqref{eqn:Psi-cocycle}, we have
\begin{equation}
    X^{\g,\sigma}(F)(v) = \frac{{\rm d}}{{\rm d}t}\bigg|_{t=0} {\rm tr}\,U_{\varPhi^{(\g,\sigma)}_s(t,v)}(0) = \frac{{\rm d}}{{\rm d}t}\bigg|_{t=0} {\rm tr}\,U_v(t) = {\rm tr}\,\dot{U}_v(0).
  \end{equation}
  As the magnetic flow is volume-preserving \cite[Proposition 1.5]{Wojtkowski_2000}, it follows that 
  the integral of ${\rm tr}(\dot{U}_v(0))$ over $\Sigma_s$ vanishes,
  and so \eqref{eqn:almost-there} leads us to
  \begin{equation}\label{eqn:inequality_Hopf}
    \int_{SM}{\rm Ric}^\Omega_s(v)\,{\rm d}\mu_{\g, 1}(v) =  -s^{1-n}\int_{\Sigma_s} {\rm tr}(U_v(0)^2)\,{\rm d}\mu_{\g,s}(v) \leq 0,
  \end{equation}concluding the first part of the argument.

 Assuming now that the equality holds in \eqref{eqn:inequality_Hopf}, measurability and non-ne\-ga\-ti\-vi\-ty of the function $\Sigma_s\ni v\mapsto {\rm tr}(U_v(0)^2)$ tells us that ${\rm tr}(U_v(0)^2) = 0$ for almost every $v\in \Sigma_s$. As $U_v(0)$ is self-adjoint, it must then be the case that $U_v(0) = 0$ for almost every $v\in \Sigma_s$. Using that the flow is volume preserving along with \eqref{eqn:conjugation-type}, we see that for every $t \in \mathbb{R}$ the set $B_t = \{v \in \Sigma_s \ | \ U_v(t) = 0\}$ has full $\mu_{\g, s}$-measure. Hence, $B = B_0 \cap \left( \bigcap_{n \geq 1} B_{1/n} \right)$ also has full measure, so for almost every $v \in \Sigma_s$ we have $\dot{U}_v(0) =\lim_{n \rightarrow \infty} n[U_v(1/n) - U_v(0)] = 0.$
Substituting this into \eqref{eqn:Riccati-t} with $t=0$, we obtain that $M^\Omega_s = 0$ as required.
 
\section{The meaning of magnetic flatness and proof of Theorem \ref{thm:mag_flat}}\label{sec:mag_flatness}

Recall that a K\"{a}hler manifold may be defined as a Riemannian manifold $(M, \g)$ equipped with an almost-complex structure $J$ such that $\g(J\cdot,J\cdot) = \g$ and ${\nabla J = 0}$, with $\nabla$ denoting the Levi-Civita connection of $(M,\g)$. This latter condition then implies that $J$ is integrable. Restricting the sectional curvature function of $(M,\g)$ to $J$-invariant planes, one has the well-known notion of \emph{holomorphic sectional curvature}, which may be regarded as the smooth function ${\rm sec}^{\rm hol}\colon SM\to \R$ given by ${\rm sec}^{\rm hol}(v) = {\rm sec}(v,Jv)$. Note that ${\rm sec}^{\rm hol}$ also equals the Gaussian curvature of $(M,\g)$ when $\dim M=2$.

Given any $\lambda \in (0,\infty)$, we may consider the magnetic system $(\g, \lambda\omega)$ on $M$, where $\omega = \g(J\cdot,\cdot)$ is the K\"{a}hler form of $(M,\g)$. Systems of this form are called \emph{K\"{a}hler magnetic systems}, and they have been extensively studied by Adachi \cite{Adachi_1995, Adachi_1997, Adachi_2012, Adachi_2014, Adachi_2019}.

Whenever the holomorphic sectional curvature of $(M,\g)$ is constant and equal to $k$, the Riemann curvature tensor is given by
\begin{equation}\label{eqn:Riemann_Kahler_H}
  R(X,Y)Z = \frac{k}{4}(\langle Y, Z\rangle X - \langle X,Z\rangle Y - \langle Y, JZ\rangle JX + \langle X, JZ\rangle JY + 2\langle X,JY\rangle JZ),
\end{equation}for all vector fields $X$, $Y$, and $Z$ on $M$ (cf. \cite[Proposition 7.3, p. 167]{KN2}). In this case, a straightforward computation shows that
\begin{equation}\label{eqn:magsec_Kahler}
  ({\rm sec}^\Omega_s)_v(w) = \frac{s^2k+\lambda^2}{4} (1+3\langle v,Jw\rangle^2),
\end{equation}for all $s>0$ and all pairs of orthonormal vectors $v,w$ tangent to $M$. In the case where $\dim M=2$, \eqref{eqn:magsec_Kahler} reads ${\rm sec}_s^\Omega\equiv s^2k+\lambda^2$. 

Our Theorem \ref{thm:mag_flat} serves as a rough ``converse'' of the above discussion.

\begin{proof}[Proof of Theorem \ref{thm:mag_flat}]
  Combining \eqref{eqn:AOmega} with the easily-verified relation \begin{equation}({\rm sec}^\Omega_s)_v(w) = s^2 {\rm sec}(v,w) - s \langle (\nabla_w\Omega)(v),w\rangle + \big\langle (A^\Omega)_v(w),w\big\rangle,\end{equation} the constant ${\rm sec}^\Omega_s$ condition can be written as
  \begin{equation}\label{eqn:mag_flatness}
    s^2{\rm sec}(v,w) - s\langle (\nabla_w\Omega)(v),w\rangle +\frac{3}{4}\langle v,\Omega(w)\rangle^2 + \frac{1}{4} \|\Omega(w)\|^2  = c,
  \end{equation}for all pairs of orthonormal vectors $v,w\in TM$. We first claim that
  \begin{equation}\label{eqn:fake_nearly_kahler}
    \parbox{.43\textwidth}{$(\nabla_w\Omega)(w) = 0$, for every $w\in TM$.}
  \end{equation}To establish \eqref{eqn:fake_nearly_kahler}, we may assume without loss of generality that $\|w\|=1$. Replacing $v\mapsto -v$ in \eqref{eqn:mag_flatness} and invoking skew-adjointness of $\nabla_w\Omega$ immediately yields $\langle (\nabla_w\Omega)(w),v\rangle = 0$, for all unit vectors $v$ orthogonal to $w$. This forces $(\nabla_w\Omega)(w)$ to be proportional to $w$, while at the same time $(\nabla_w\Omega)(w)$ is orthogonal to $w$; \eqref{eqn:fake_nearly_kahler} thus holds.

  Polarizing \eqref{eqn:fake_nearly_kahler}, it follows that $\nabla\Omega$ is skew-symmetric, and as a consequence $\nabla\sigma$ is alternating. However, the exterior derivative of a differential form equals the alternator of its covariant differential, up to a nonzero dimensional constant depending on conventions. Therefore, $\nabla\sigma = 0$ follows from ${\rm d}\sigma = 0$.

  As a nontrivial parallel tensor field on any connected manifold equipped with an affine connection is necessarily nowhere-vanishing, we now claim that
  \begin{equation}\label{eqn:Omega-conformal}
    \parbox{.538\textwidth}{$\Omega$ is a conformal endomorphism of $TM$, and thus $\,\|\Omega\|^{-1}\Omega\,$ is a complex structure on $\,M$.}
  \end{equation}
  Indeed, replacing $(v,w) \mapsto (w,v)$ in \eqref{eqn:mag_flatness} gives us that $\|\Omega(v)\|=\|\Omega(w)\|$ for all pairs of orthonormal vectors $v,w$ tangent to $M$, and such condition immediately implies that $\Omega$ is conformal, with constant conformal factor $\|\Omega\|>0$. Namely, we have that $\langle \Omega(v),\Omega(w)\rangle = \|\Omega\|^2\langle v,w\rangle$ for all vectors $v,w$ tangent to $M$. As $\Omega$ is skew-adjoint, nondegeneracy of $\g$ now implies that $\Omega^2 = -\|\Omega\|^2{\rm Id}$, yielding \eqref{eqn:Omega-conformal}.

  Since $\Omega$ is parallel, we have that $J=\|\Omega\|^{-1}\Omega$ is parallel and preserves $\g$, so 
  \eqref{eqn:Omega-conformal} makes $(M,\g)$ a K\"{a}hler manifold. Let ${\rm sec}^{\rm hol}$ be its holomorphic sectional curvature. Setting $w= Jv$ in \eqref{eqn:mag_flatness} leads to $s^2{\rm sec}^{\rm hol}(v) + \|\Omega\|^2 = c$ --- yielding \eqref{thm:mag_flat3} --- and so \eqref{eqn:magsec_Kahler} with $\lambda = -\|\Omega\|$ implies that
  \begin{equation}\label{eqn:kill_c}
    \frac{c}{4}(1+3\langle v,Jw\rangle^2) = c,
  \end{equation}for all pairs of orthonormal vectors $v,w$ tangent to $M$. Now, it follows from \eqref{eqn:kill_c} that $c=0$ whenever $\dim M>2$. Indeed, if $c\neq 0$, then $1+3\langle v,Jw\rangle^2=4$ leads to $|\langle v,Jw\rangle|=1$, and the equality case in the Cauchy-Schwarz inequality implies that
  \begin{equation}\label{eqn:CS_equality}
    \parbox{.57\textwidth}{the set $\{v,Jw\}$ is linearly dependent, for all pairs of orthonormal vectors $v,w$ tangent to $M$.}
  \end{equation}If $\dim M\geq 3$, \eqref{eqn:CS_equality} easily leads to a contradiction by fixing $w$ and choosing two mutually orthogonal unit vectors $v,v'$, both orthogonal to $w$.
 
\end{proof}


\appendix

\section{Negative magnetic curvature implies Anosov} \label{appendixAnosov}

The goal of this appendix is to prove the following:
\begin{prop} \label{lem:anosov}
	If $\sec^\Omega_s < 0$, then $\varPhi_s^{(\g, \sigma)}$ is Anosov.
\end{prop}

Without loss of generality, assume $s = 1$ and write $\varphi^t(v) = \Phi^{(\g, \sigma)}_1(t,v)$. Recall that $\varphi^t$ is \emph{Anosov} if there exists a $\varphi^t$-invariant splitting $T SM = E^+ \oplus E^0 \oplus E^-$ and constants $a,b > 0$ so that 
\begin{enumerate}[(1)]
	\item $E^0_v$ is spanned by $X^{\g, \sigma}_v$ for all $v \in SM$, and
	\item $\|{\rm d} (\varphi^{\pm t})_v(\xi)\| \leq b {\rm e}^{-ta} \|\xi\|$ for all $v \in SM$, $\xi \in E^{\pm}_v$, and $t \geq 0$, .
\end{enumerate} 
Since $M$ is compact, note that the choice of metric on $SM$ does not matter. We will consider $SM$ equipped with the \emph{twisted Sasaki metric}, so that the norm of $\xi \in T_vSM$ is given by 
\begin{equation}\label{eqn:twisted_sasaki}
	\|\xi\|^2 = \|\pi_*(\xi)\|^2 + \|K^\sigma(\xi)\|^2.
\end{equation}
Let $Q$ be the quotient bundle equipped with the quotient norm coming from \eqref{eqn:twisted_sasaki}, and let $\varPsi^{(\g, \sigma)}_t$ be the quotient flow introduced in \eqref{eqn:induces_symplec}.  We recall \cite[Proposition 5.1]{Wojtkowski_2000}, which allows us to work on the level of the quotient bundle instead of $TSM$:
\begin{lem} \label{lem:wojtkowski}
	The flow $\varphi^t$ is Anosov if and only if there is a $\varPsi^{(\g, \sigma)}_t$-invariant splitting $Q = E^+ \oplus E^-$ and two positive constants $a,b$ so that we have $\|\varPsi^{(\g,\sigma)}_{\pm t}(\xi)\| \leq b {\rm e}^{-ta} \|\xi\|$ for all $v \in SM$, $\xi \in E^{\pm}_v$, and $t \geq 0$.
\end{lem}

Next, we consider two associated matrix differential equations and recall some standard comparison theory. Since $\sec^\Omega_1 < 0$, we have that the operators $S^\pm_v$ are defined by Proposition \ref{prop:negative_sec_no_conj} and Lemma \ref{lem:lim_Sv_exist}. Consider the maps $Y_v^\pm(t) : v^\perp \rightarrow v^\perp$ given by $Y_v^\pm(t)w = P_{-t}(J_w^\pm(t)^\perp)$, where $J_w^\pm \in \mathcal{J}^{\g, \sigma, n}_0(v)$ has initial conditions $J_w^\pm(0) = w$ and $(\tD J_w^\pm)(0) = S_v^\pm w$.
Defining \hbox{$A_v(t) = P_{-t} \circ (M^\Omega_1)_{\varphi^t(v)} \circ P_t$} and using \eqref{eqn:Jacobi2}, we have that the Jacobi equation \hbox{$\ddot{Y}_v^{\pm}(t) + A_v(t) \circ Y_v^{\pm}(t) = 0$} holds on \hbox{$v^\perp \cong \mathbb{R}^{n-1}$}. Similarly, defining $U_v^\pm(t) = \dot{Y}_v^\pm(t) \circ (Y_v^\pm(t))^{-1}$ and using \eqref{eqn:Riccati-t}, we have that the Riccati equation \hbox{$\dot{U}_v^\pm(t) + (U_v^\pm(t))^2 + A_v(t) = 0$} holds on \hbox{$v^\perp \cong \mathbb{R}^{n-1}$}.

Under the assumption that $\sec^\Omega_1 < 0$, we can find two positive constants $\tilde{a}, a$ so that $- \tilde{a}^2 \text{Id}_{v^\perp} \leq A_v(t) \leq -a^2 \text{Id}_{v^\perp}$ for all $v \in SM$ and $t \in \mathbb{R}$, where here we use the Loewner order described in Section \ref{sec:mag_green_bundles}. These constants allow us to apply Riccati comparison results in order to get bounds on $U_v^\pm(t)$. The lower bound $-\tilde{a}^2 \text{Id}_{v^\perp} \leq A_v(t)$ gives $|\langle U_v^\pm(t)w, w \rangle| \leq \tilde{a}\|w\|^2$ for all $v \in SM$, $w \in v^\perp$, and $t \in \mathbb{R}$ \cite[Lemma 4.4]{Gouda_1997}. Moreover, since $U_v^\pm(t)$ is a symmetric matrix, 
we deduce
\begin{equation} \label{eqn:ric_comp}
	\|U_v^\pm(t)w\| \leq \tilde{a} \|w\|
\end{equation}
for all $v \in SM$, $w \in v^\perp$, and $t \in \mathbb{R}$.

Similarly, the upper bound $A_v(t) \leq -a^2 \text{Id}_{v^\perp}$ gives  
\begin{equation} \label{eqn:ric_comp2}
	\langle U_v^+(t)w, w \rangle \leq - a  \quad \text{and} \quad \langle U_v^-(t)w, w \rangle \geq a.
\end{equation}
for all $v \in SM$, $w \in v^\perp$ a unit vector, and $t \in \mathbb{R}$ \cite[Lemma 4.5]{Gouda_1997}. In particular, 
\begin{equation} \label{eqn:exp}
	\begin{split}
		\|Y_v^+(t)w\| \leq {\rm e}^{-at}\|w\| & \text{ and }  \|Y_v^-(t)w\| \geq {\rm e}^{at}\|w\|, 
	\end{split}
\end{equation}
for all $v \in SM$, $w \in v^\perp$, and $t \geq 0$ \cite[Corollary 4.6]{Gouda_1997}. Note that applying the Cauchy-Schwarz inequality to \eqref{eqn:ric_comp2} yields \hbox{$a \leq  \langle U_v^-(t)w, w \rangle \leq \|U_v^-(t)w\|$}, which we can rewrite as
\begin{equation} \label{eqn:ric_comp3}
	a \|w\| \leq \|U_v^-(t)w\| 
\end{equation}
for all $v \in SM$, $w \in v^\perp$, and $t \in \mathbb{R}$.

Finally, it will be useful to understand how the quotient norm interacts with Jacobi fields. For $v \in SM$, let $\xi \in Q_v$ and let $J(t) \in \mathcal{J}^{\g, \sigma, n}_0(v)$ be the corresponding Jacobi field given by Lemma \ref{lem:quotient_correspondence}. Similar to the geodesic case, we have the relations $\pi_*(\varPsi^{(\g, \sigma)}_t(\xi)) = J(t)^\perp$ and $K^\sigma(\varPsi^{(\g, \sigma)}_t(\xi)) = \tD J(t)$ (cf.\ \cite[Lemma 5.3]{Gouda_1997} and \cite[Lemma 1.40]{Paternain99_book}). Thus, \hbox{$\| \varPsi^{(\g, \sigma)}_t(\xi)\|^2 = \|J(t)^\perp\|^2 + \|\tD J(t)\|^2$} by \eqref{eqn:twisted_sasaki}. Letting \hbox{$w = J(0)$} and assuming that $\xi \in  E_v^{\sigma, \pm}$, we can use the $\tD$-parallel transport to rewrite the norm in terms of $Y_v^\pm(t)$ as follows:
\begin{equation} \label{eqn:twisted_sasaki_2}
	\| \varPsi^{(\g, \sigma)}_t(\xi)\|^2 = \|Y_v^\pm(t)w\|^2 + \|\dot{Y}_v^\pm(t)w\|^2.
\end{equation}

\begin{proof}[Proof of Proposition \ref{lem:anosov}]
	The first step is to show that the bundles $E^{\sigma, +}$ and $E^{\sigma, -}$ constructed in Lemma \ref{lem:lim_Sv_exist} give us the $\varPsi^{(\g, \sigma)}_t$-invariant splitting needed to apply Lemma \ref{lem:wojtkowski}. Since $\dim(E^{\sigma, +}_v) = \dim(E^{\sigma, -}_v) = n-1$ for all $v \in SM$, it suffices to show that $E^{\sigma,+}_v \cap E^{\sigma,-}_v = 0$ for all $v \in SM$. Using the correspondence described in Lemma \ref{lem:quotient_correspondence}, let $J \in E^{\sigma, +}_v \cap E^{\sigma, -}_v$ and let $J(0) = w \in v^\perp$. We can simultaneously write $J(t) = Y_v^+(t)w$ and $J(t) = Y_v^-(t)w$, and \eqref{eqn:exp} forces $J = 0$. 
	
	The next step is to show that we have the necessary exponential bounds. We start by examining $E^{\sigma, +}$. Let $v \in SM$, $\xi \in E^{\sigma, +}_v$, and $t \geq 0$. Combining \eqref{eqn:exp} and \eqref{eqn:twisted_sasaki_2}, we have $	\|\varPsi^{(\g, \sigma)}_t(\xi)\|^2 \leq {\rm e}^{-2at}\|w\|^2 + \|\dot{Y}_v^+(t)w\|^2.$
	Using \eqref{eqn:ric_comp} and \eqref{eqn:exp}, we see that $\|\dot{Y}_v^+(t) w\| = \|U_v^+(t) \circ Y_v^+(t)w \| \leq \tilde{a} \|Y_v^+(t)\| \leq \tilde{a} {\rm e}^{-at}\|w\|$, and thus
	\begin{equation} \label{eqn:anosov1}
		\|\varPsi^{(\g, \sigma)}_t(\xi)\| \leq \sqrt{1+\tilde{a}^2} \ {\rm e}^{-at}\|\xi\|.
	\end{equation}
	
	We now examine $E^{\sigma, -}$. Let $v \in SM$, $\xi \in E^{\sigma, -}_v$, and $t \geq 0$. Again, combining  \eqref{eqn:exp} and \eqref{eqn:twisted_sasaki_2} yields \hbox{$\|\varPsi^{(\g, \sigma)}_{t}(\xi)\|^2 \geq {\rm e}^{2at}\|w\|^2 + \|\dot{Y}_v^-(t)w\|^2$}. Using \eqref{eqn:exp} and \eqref{eqn:ric_comp3}, we see that $\|\dot{Y}_v^-(t)w\| = \|U_v^-(t) \circ Y_v^-(t)w\| \geq a \|Y_v^-(t)w\| \geq a {\rm e}^{at}\|w\|$, and thus we have the lower bound 
	\begin{equation} \label{eqn:anosov15}
		\|\varPsi^{(\g, \sigma)}_t(\xi)\| \geq \sqrt{1+a^2} \ {\rm e}^{at}\|w\|.
	\end{equation} 
	To finish, observe that $\|\xi\|^2 = \|w\|^2 + \|\dot{Y}_v^-(0)w\|^2$ by \eqref{eqn:twisted_sasaki_2}. Notice that we can use \eqref{eqn:ric_comp} to get \hbox{$\|\dot{Y}_v^-(0)w\| = \|U_v^-(0)w\| \leq \tilde{a} \|w\|$}, and from this we can deduce the bound $(1 + \tilde{a}^2)^{-1/2}\|\xi\| \leq \|w\|$. Combining this with \eqref{eqn:anosov15} and using that the inequality holds for all $v \in SM$, $\xi \in E^{\sigma, -}_v$ and $t \geq 0$, we have
	\begin{equation} \label{eqn:anosov2}
		\|\varPsi^{(\g, \sigma)}_{-t}(\xi)\| \leq \sqrt{\frac{1 + \tilde{a}^2}{1+a^2}} \ {\rm e}^{-at}\|\xi\|.
	\end{equation}
	Setting $b = \max\{ (1+\tilde{a}^2)^{1/2}, (1 + \tilde{a}^2)^{1/2} (1+a^2)^{-1/2}\}$, we see that Lemma \ref{lem:wojtkowski}, \eqref{eqn:anosov1}, and \eqref{eqn:anosov2} prove the result.
\end{proof}

\section{The Ma\~{n}\'{e} critical value of K\"{a}hler magnetic systems} \label{appendixA}

The goal of this appendix is to mimic the proof of \cite[Lemma 6.1]{CieliebakFrauenfelderPaternain_2010} in order to establish \eqref{eqn:mane2}. Namely:

\begin{prop}\label{lem:appendix}
	Let $(M,\g)$ be a compact K\"{a}hler manifold with K\"{a}hler form $\omega$, whose holomorphic sectional curvature is constant and equal to $k<0$. Then for every constant $\lambda\in \R$, the Ma\~{n}\'{e} critical value of the K\"{a}hler magnetic system $(\g,\lambda\omega)$ is given by $s_0(\g,\lambda\omega) = |\lambda|/\sqrt{-k}$.
\end{prop}

Here, we will use an alternative description of the Ma\~{n}\'{e} critical value (see \cite[Section 5.1]{CieliebakFrauenfelderPaternain_2010}). Let $(\g,\sigma)$ be any magnetic system on $M$, $s>0$, and $p\colon\widetilde{M}\to M$ be the universal covering of $M$. If $p^*\sigma$ is exact, we consider the Lagrangian $L_s\colon T\widetilde{M}\to \R$ given by
\begin{equation}
	L_s(v) = \frac{1}{2} (\|v\|^2+s^2) - \theta(v),
\end{equation}
where $\|\cdot\|$ is the norm induced by $p^*\g$ and $\theta$ is any primitive of $p^*\sigma$. Writing $\Lambda\widetilde{M}$ for the space of absolutely continuous loops in $\widetilde{M}$, we also consider the action functional $A_s\colon \Lambda\widetilde{M}\to \R$ given by
\begin{equation}\label{eqn:action}
	A_s(\gamma) = \int_0^T L_s(\dot{\gamma}(t))\,{\rm d}t,
\end{equation}where $T>0$ is the smallest period of $\gamma$. Evidently, \eqref{eqn:action} is independent of the choice of primitive $\theta$. With this in place, we have that
\begin{equation}\label{eqn:mane-alternative}
	s_0(\g,\sigma) = \begin{cases} {\displaystyle \inf\{s>0 \mid A_s(\gamma)\geq 0\mbox{ for all }\gamma\in\Lambda\widetilde{M} \}}  & \text{if } p^* \sigma \text{ is exact}, \\ +\infty &\text{otherwise}. \end{cases}
\end{equation}

\begin{proof}[Proof of Proposition \ref{lem:appendix}]
	As it follows from \eqref{eqn:mane} that
	\begin{equation}
		s_0(-k\g,\lambda\omega) = \frac{|\lambda|}{\sqrt{-k}}s_0(\g,\omega),
	\end{equation}it suffices to prove that $s_0(\g,\omega)=1$ under the assumption that $k=-1$. By \cite[Theorems 7.8 and 7.9]{KN2}, the universal covering of $M$ is holomorphically isometric to the unit ball $\mathbb{B}^n = \{z\in \mathbb{C}^n \mid \|z\|<1\}$ equipped with the K\"{a}hler metric
	\begin{equation}\label{eqn:g-ball}
		\widetilde{\g} = -4\sum_{j,k=1}^n \left(\frac{(1-\|z\|^2)\delta_{jk} + \bar{z}_jz_k}{(1-\|z\|^2)^2}\right)\,{\rm d}z_j\,{\rm d}\bar{z}_k.
	\end{equation}
	The K\"{a}hler form $\widetilde{\omega}$ of $(\mathbb{B}^n, \widetilde{\g})$ is given by
	\begin{equation}\label{eqn:omega-ball}
		\widetilde{\omega} = -4{\rm i} \sum_{j,k=1}^n\left(\frac{(1-\|z\|^2)\delta_{jk} + \bar{z}_jz_k}{(1-\|z\|^2)^2}\right)\,{\rm d}z_j\wedge{\rm d}\bar{z}_k,
	\end{equation}
	and admits
	\begin{equation}\label{eqn:theta-ball}
		\theta = \frac{-2{\rm i}}{1-\|z\|^2}\sum_{j=1}^n (\bar{z}_j\,{\rm d}z_j - z_j\,{\rm d}\bar{z}_j)
	\end{equation}as a primitive. A routine computation using \eqref{eqn:g-ball}, \eqref{eqn:theta-ball}, and \eqref{eqn:mane}, shows that $\|\theta_z\| = \|z\| < 1$ for all $z\in \mathbb{B}^n$, so that $s_0(\g,\omega) \leq \sup_{z\in \mathbb{B}^n} \|\theta_z\| = 1$. It remains to establish the reverse inequality. To do so, we consider the totally geodesic embedding $\iota\colon \mathbb{B}^1 \to \mathbb{B}^n$ given by $\iota(\zeta) = (\zeta,0,\ldots, 0)$. From \cite[Theorem 3.19]{Goldman_1999}, we have
	\begin{equation}
		\iota^*\widetilde{\g} = -\frac{4\,{\rm d}\zeta\,{\rm d}\bar{\zeta}}{(1-|\zeta|^2)^2}\quad\mbox{and}\quad \iota^*\widetilde{\omega} = -\frac{4{\rm i}\,{\rm d}\zeta\wedge{\rm d}\bar{\zeta}}{(1-|\zeta|^2)^2}.
	\end{equation}Let $s>0$. For each $r>0$, let $\gamma_r\colon [0,\sinh r]\to \mathbb{B}^1$ be the circle of radius $r$ with respect to $\iota^*\widetilde{\g}$, parametrized with constant speed $s$. Denoting by $\ell(\gamma_r)=2\pi\sinh r$ and ${\rm a}(D_r) = 2\pi(\cosh r-1)$ the $\iota^*\widetilde{\g}$-length of $\gamma_r$ and the $\iota^*\widetilde{\g}$-area of the disk bounded by $\gamma_r$, respectively, we proceed to compute the action \eqref{eqn:action} with the aid of $\theta$ in \eqref{eqn:theta-ball} and Stokes' theorem:
	\begin{equation}\label{eqn:action_ball}
		\begin{split}
			A_s(\gamma_r)  &= \int_0^{\ell(\gamma_r)/s} \frac{1}{2} \left((\iota^*\widetilde{\g})(\dot{\gamma}_r(t),\dot{\gamma}_r(t)) + s^2\right) - (\iota^*\theta)(\dot{\gamma}_r(t)) \,{\rm d}t \\ &= s\ell(\gamma_r) - {\rm a}(D_r) \\ &= 2\pi(s\sinh r - \cosh r  +1).
		\end{split}
	\end{equation}
	It follows from \eqref{eqn:action_ball} that if $s<1$, the value $A_s(\gamma_r)$ becomes negative for sufficiently large $r$. We conclude from \eqref{eqn:mane-alternative} that $s_0(\g,\omega) \geq 1$, as desired.
\end{proof}

\section*{Conflict of interest}

The authors have no relevant financial or non-financial interests to disclose.

\bibliography{magnetic_refs}{}	
\bibliographystyle{plain}

\end{document}